\documentclass[runningheads]{llncs}
\usepackage{graphicx}
\usepackage{appendix}
\usepackage{amsmath}
\usepackage{amssymb}
\usepackage{amsfonts}
\usepackage{hyperref}
\usepackage{tikz-cd}
\usepackage{amsthm}
\usepackage{algorithmic}
\usepackage{graphicx}
\usepackage[ruled,vlined]{algorithm2e}
\usepackage{tabu}
\theoremstyle{definition}
\usepackage{mathtools}

\newtheorem{defn}{Definition}

\titlerunning{Sum-Product Bounds in Dimension 16}
\begin{document}
\title{Sum-Product Bounds \& an Inequality for the
Kissing Number in Dimension 16}
%
%
\author{Andrew Mendelsohn}
\authorrunning{A. Mendelsohn}
%
\institute{Department of EEE, Imperial College London, London, SW7 2AZ, United Kingdom. \\ \href{andrew.mendelsohn18@imperial.ac.uk}{andrew.mendelsohn18@imperial.ac.uk}}
\maketitle
\begin{abstract}
We obtain an inequality for the kissing number in 16 dimensions. We do this by generalising a sum-product bound of Solymosi and Wong for quaternions to a semialgebra in dimension 16. In particular,
we obtain the inequality 
$$k_{16}\geq \frac{\sum_{x \in \mathcal{R}}\left|\mathcal{S}_{x}\right|}{\left|\bigcup_{x \in \mathcal{R}} \mathcal{S}_{x}\right|}-1,$$
where $k_{16}$ is the 16-dimensional kissing number, and $S_x$ and $\mathcal{R}$ are sets defined below. Along the way we also obtain a sum-product bound for subsets of the octonions which are closed under taking inverses, using a similar strategy to that used for the quaternions. We use the fact that the kissing number in eight dimensions is 240 to achieve the result. Namely, we obtain the bound $$\operatorname{max}(|\mathcal{A}+\mathcal{A}|,|\mathcal{A}\mathcal{A}|)\geq \frac{|\mathcal{A}|^{4/3}}{(1928\cdot\lceil \operatorname{log}|\mathcal{A}|\rceil)^{1/3}},$$ where $\mathcal{A}$ is a finite set of octonions such that if $x\in\mathcal{A}$, $x^{-1}\in\mathcal{A}$ also.
\end{abstract}
%
%
\section{Introduction}
Given an arbitrary subset $\mathcal{A}$ of a ring, let $\mathcal{A}\mathcal{A}$ be the set of products of elements of $\mathcal{A}$, and $\mathcal{A}+\mathcal{A}$ the set of sums of elements of $\mathcal{A}$; how large can both $|\mathcal{A}\mathcal{A}|$ and $|\mathcal{A}+\mathcal{A}|$ be? In \cite{Erdos}, Erd{\H{o}}s and Szemer{\'e}di showed that for subsets of the integers, $\max (|\mathcal{A}+\mathcal{A}|,|\mathcal{A} \mathcal{A}|) \gg|\mathcal{A}|^{1+\delta}$ for some $\delta>0$. For the quaternions, Chang \cite{chang} showed that $\delta\geq1/54$, and Solymosi and Wong \cite{kissingsumprods} showed $\delta$ can be arbitrarily close to $1/3$. The state of the art for the exponent for subsets of the quaternions is the work of Basit and Lund \cite{improvedquatbd}, in which the authors show $\delta>1/3$.\\
\indent In the first part of this work, we observe that the method of \cite{kissingsumprods} adapts to the octonions, for certain conditioned subsets $\mathcal{A}$. In particular, for a subset $\mathcal{A}\subset\mathbb{O}\setminus \{0\}$, let $e_1,...,e_{|\mathcal{A}/\mathcal{A}|}$ be the elements of $\mathcal{A}/\mathcal{A}:= \mathcal{A}\mathcal{A}^{-1}\cap\mathcal{A}^{-1}\mathcal{A}$, $a_i$ be the number of ways of multiplying an element of $\mathcal{A}$ with an element of $\mathcal{A}^{-1}$ (in that order) to obtain $e_i$, and $b_i$ be the number of ways of multiplying an element of $\mathcal{A}^{-1}$ with an element of $\mathcal{A}$ (in that order) to obtain $e_i$, for $i=1,...,|\mathcal{A}/\mathcal{A}|$. Suppose the $a_i$ and $b_i$ are such that $\sum_i a_ib_i\geq \operatorname{min}(\sum_i a_i^2,\sum_i b_i^2)$, and say without loss of generality $\operatorname{min}(\sum_i a_i^2,\sum_ib_i^2)= \sum_i a_i^2$. Then $$\operatorname{max}\left(|\mathcal{A}+\mathcal{A}|,|\mathcal{A}\mathcal{A}^{-1}|\right)\gg \frac{|\mathcal{A}||\mathcal{A}/\mathcal{A}|^{1/4}}{\lceil \operatorname{log}|\mathcal{A}|\rceil^{1/4}}.$$ For subsets $\mathcal{A}$ satisfying closure under inversion, we obtain $$\operatorname{max}\left(|\mathcal{A}+\mathcal{A}|,|\mathcal{A}\mathcal{A}|\right)\gg \frac{|\mathcal{A}|^{4/3}}{\lceil \operatorname{log}|\mathcal{A}|\rceil^{1/3}},$$ that is, $\delta$ is arbitrarily close to $1/3$. To our knowledge this is the first such bound for subsets of the octonions.\\
\indent Extending this, in the second half we consider a doubling process applied to the octonions which yields a division left semialgebra\footnote{Where a left semialgebra is a vector space over a field with a binary operation for which left multiplication is linear.}, named the `16-ons' to distinguish it from the 16-dimensional algebra obtained from the Cayley-Dickson doubling process, and other `sedenionic' constructions over the reals (see, for instance, \cite{smithseds}). This space is `division' in the sense that every non-zero element has a multiplicative inverse.\\
\indent We extend the argument we applied to the octonions to hold in this space for sets subject to some conditions, which yields an inequality involving the kissing number in 16 dimensions, $k_{16}$, which rearranges to give an abstract lower bound for $k_{16}$ in terms of the ratio of sizes of particular sets of 16-ons, defined below. The result of the first part, giving a sum-product bound for octonions, can then be seen as a `warm-up' demonstrating that the work of \cite{kissingsumprods} adapts to a nonassociative setting.\\
\indent In particular, we find that, when $\mathcal{A}$ is a finite set of 16-ons which have only the first nine entries possibly non-zero and $x\in \mathcal{A}/\mathcal{A}$, letting $\ell(x)$ denote the number of representatives of $x$ in $\mathcal{A}\mathcal{A}^{-1}$, and defining $$\mathcal{R}:=\operatorname{max}_{I}\left\{x \in \mathcal{A} / \mathcal{A}: \ell(x) \geq r(x)\text{ and }2^{I} \leq \ell(x)<2^{I+1}\right\}$$ and $$\mathcal{S}_{x}:=\left\{(a+c, b+d) \in(\mathcal{A}+\mathcal{A}) \times(\mathcal{A}+\mathcal{A}): a b^{-1}=x \text { and } c d^{-1}=\phi(x)\right\},$$
where $x\in\mathcal{R}$ and $\phi: \mathcal{R} \rightarrow \mathcal{R}$ maps an element of $\mathcal{R}$ to the closest (distinct) element in $\mathcal{R}$, then
$$k_{16}\geq \frac{\sum_{x \in \mathcal{R}}\left|\mathcal{S}_{x}\right|}{\left|\bigcup_{x \in \mathcal{R}} \mathcal{S}_{x}\right|}-1.$$ 
We have not used this to concretely lower bound $k_{16}$. Unfortunately, the doubling of the 16-ons, the `32-ons', lose necessary algebraic properties for the argument to hold, so we are unable to use this method for higher dimensional bounds.

\section{Preliminaries}
Given a finite set $\mathcal{A}$ in a real division algebra $\mathbb{A}$, we define $$\mathcal{A}+\mathcal{A} = \{a+b\text{ : }a,b\in\mathcal{A}\}\text{ and }\mathcal{A}\mathcal{A} = \{a\cdot b\text{ : }a,b\in\mathcal{A}\}.$$
Similarly, we let $\mathcal{A}^{-1}$ be the set of the inverses of elements of $\mathcal{A}$. In addition, we will make reference to what has been called the \textit{ratioset} of $\mathcal{A}$: $$\mathcal{A}/\mathcal{A}:= \mathcal{A}\mathcal{A}^{-1}\cap\mathcal{A}^{-1}\mathcal{A}.$$ We also need the notion of the multiplicative energy of a set:
\begin{defn}
Let $\mathcal{A}$ be a finite set in $\mathbb{A}$. Then the \textit{multiplicative energy} of $\mathcal{A}$ is defined $$E(\mathcal{A}) = |\{(a,b,c,d)\in\mathcal{A}^4\text{ : }ca = db\}|.$$
\end{defn}
\noindent We are also interested in the following related set: $$E^\prime(\mathcal{A}) = |\{(a,b,c,d)\in\mathcal{A}^4\text{ : }ab^{-1} = c^{-1}d\}|.$$ Note that when $\mathcal{A}$ is a subset of an associative ring, these two sets coincide, and $E(\mathcal{A}) = E^\prime(\mathcal{A})$. However, this is not prima facie the case for subsets of the octonions $\mathbb{O}$.\\
\indent There is the following bound (using the Cauchy-Schwartz inequality) for associative $\mathbb{A}$:
$$E(\mathcal{A})\geq \frac{|\mathcal{A}|^4}{|\mathcal{A}\mathcal{A}|}.$$ 
We desire a similar style bound for $E^\prime(\mathcal{A})$ with $\mathcal{A}\subset \mathbb{O}$; we have the following lemma:
\begin{lemma}
Let $\mathcal{A}\subset\mathbb{O}$ and $0\not\in\mathcal{A}$. Let $e_1,...,e_{|\mathcal{A}/\mathcal{A}|}$ be the elements of $\mathcal{A}/\mathcal{A}$, $a_i$ be the number of ways of multiplying an element of $\mathcal{A}$ with an element of $\mathcal{A}^{-1}$ (in that order) to obtain $e_i$, and $b_i$ be the number of ways of multiplying an element of $\mathcal{A}^{-1}$ with an element of $\mathcal{A}$ (in that order) to obtain $e_i$, for $i=1,...,|\mathcal{A}/\mathcal{A}|$. Suppose the $a_i$ and $b_i$ are such that $\sum_i a_ib_i\geq \operatorname{min}(\sum_i a_i^2,\sum_i b_i^2)$. Then $$E^\prime(\mathcal{A})\geq \frac{|\mathcal{A}|^4|\mathcal{A}/\mathcal{A}|}{\operatorname{min}\left(|\mathcal{A}\mathcal{A}^{-1}|,|\mathcal{A}^{-1}\mathcal{A}|\right)^2}.$$ When $\mathcal{A}$ is closed under taking inverses, $E^\prime(\mathcal{A})\geq \frac{|\mathcal{A}|^4}{|\mathcal{A}\mathcal{A}|}$.
\end{lemma}
\begin{proof}
Without loss of generality, suppose $\operatorname{min}(\sum_i a_i^2,\sum_i b_i^2) = \sum_i a_i^2$. Let $m:=|\mathcal{A}/\mathcal{A}|$. Observe $\sum_{i=1}^{m}a_i = \frac{|\mathcal{A}||\mathcal{A}^{-1}||\mathcal{A}/\mathcal{A}|}{|\mathcal{A}\mathcal{A}^{-1}|} = \frac{|\mathcal{A}|^2|\mathcal{A}/\mathcal{A}|}{|\mathcal{A}\mathcal{A}^{-1}|}$. Then $E^\prime(\mathcal{A}) = \sum_{i=1}^{m}a_ib_i\geq \sum_{i=1}^{m}a_i^2 = \|(a_1,...,a_{m})\|^2\|(1/\sqrt{m},...,1/\sqrt{m})\|^2\geq \frac{(\sum_{i=1}^m a_i)^2}{m} = \frac{|\mathcal{A}|^4|\mathcal{A}/\mathcal{A}|^2}{|\mathcal{A}/\mathcal{A}||\mathcal{A}\mathcal{A}^{-1}|^2} = \frac{|\mathcal{A}|^4|\mathcal{A}/\mathcal{A}|}{|\mathcal{A}\mathcal{A}^{-1}|^2}$, by Cauchy-Schwartz. When $\mathcal{A}=\mathcal{A}^{-1}$, the inequality simplifies to the stated result. 
\end{proof}
\subsubsection{Complex quadrants, quaternionic hexadecants, and octonionic hexapentacontadictants} One can divide up the complex plane into quadrants, and consider the case when a given complex number $a+ib$ has both its coefficients positive, or both its coefficients negative. This corresponds to the complex number lying in the top right or bottom left quadrant of the plane. Moreover, one can say that given two complex numbers, both \textit{lie in the same quadrant} if they both have all positive or all negative coefficients. Then the norm of the sum of any two such complex numbers is larger than the norm of either number alone.\\
\indent In a similar way one can divide up the quaternions, depending on whether a given quaternion has all its coefficients positive or all negative (or not). In this case, following \cite{kissingsumprods}, one says that two quaternions \textit{lie in the same hexadecant} (there being $16 = 2^4$ combinations of the vector of signs of the coefficients for a given quaternion) if they both have all positive, or all negative, coefficients.\\
\indent Finally, we will make the same requirement for the octonions: two octonions \textit{lie in the same hexapentacontadictant} if they both have all positive or all negative coefficients. Then the norm of the sum of any two such octonions is larger than the individual norms of the elements summed.

\subsubsection{The 16-ons} The Cayley-Dickson process is a method to take an algebra $\mathbb{A}_{n-1}$, and create a new algebra $\mathbb{A}_{n}$ with twice the dimension of the old. When applied to $\mathbb{A}_{0}=\mathbb{R}$, we successively get $\mathbb{A}_{1}=\mathbb{C}$, $\mathbb{A}_{2}=\mathbb{H}$, and $\mathbb{A}_{3}=\mathbb{O}$, iterating three times. These are normed division algebras. Hurwitz showed these are the only normed real division algebras; later, Pfister \cite{pfister} gave an identity for writing the product of two sums of 16 squares as the sum of 16 rational squares, essentially showing that norms on real algebras exist in higher dimensions. In particular, in any power of two dimension, there exists an algebra with a norm (which satisfies properties sometimes referred to as being `nicely normed'). However, $\mathbb{A}_{n}$ is not a division algebra once $n>3$, so the norm is not multiplicative: in particular, since zero-divisors exist, one might have $\|ab\| = 0 \ne \|a\|\cdot \|b\|$.\\
\indent It was observed by Conway and Smith that, tweaking the Cayley-Dickson construction, one can create a doubling process that yields multiplicative norms. However, since the only normed division algebras have $n<4$, once at $n=4$ one starts paying for multiplicativity of the norm with other properties: for example, alternativity (that is, $x(xy)=x^2y)$ is lost at this point.\\
\indent The algebras obtained from this adjusted Cayley-Dickson process have been studied in depth by Smith \cite{Smith16ons} (it seems with some collaboration with Conway; this process has been referred to as the `Conway-Smith doubling process' in \cite{LundstromHilbert90}), and called $2^n$-ons. Given two elements of $\mathbb{A}_{n-1}$, we define $\mathbb{A}_n$ as having additive group $\mathbb{A}_{n-1}\times\mathbb{A}_{n-1}$ equipped with the multiplication rule $$(a, b)(c, d)=\left(a c-\overline{b \bar{d}}, \overline{\bar{b} \bar{c}}+\overline{\bar{b} \overline{\bar{a} \overline{\overline{b^{-1}} \bar{d}}}}\right),$$
where $\overline{\cdot}$ is conjugation on $\mathbb{A}_{n-1}$. For the $2^n$-ons, conjugation can be written $\overline{(a,b)} = (\overline{a},-b)$, with $a,b\in\mathbb{A}_{n-1}$. When $n=4$, that is, in the case of the 16-ons, the above simplifies to $$(a, b)(c, d)=\left(a c-d \bar{b}, c b+\bar{a} b^{-1} \cdot b d\right)$$
When $b=0,$ the multiplication is defined $(a,0)(c,d) = (ac,\Bar{a}d)$. We list the basic properties of the $2^n$-ons.
\begin{proposition}\cite[Theorem 15]{Smith16ons}
If $a,b,x,y\in \mathbb{A}_{n}$, $n \geq 0$, and $s,t\in\mathbb{R}$, then
\begin{enumerate}
\item $A$ unique $2^{n}$-on 1 exists, with $1 x=x 1=x$.
\item $A$ unique $2^{n}$-on 0 exists with $0+x=x+0=x$ and
$0 x=x 0=0.$
\item $x+y=y+x,(x+y)+z=x+(y+z)$;
$-x$ exists with $x+(-x)=x-x=0$.
\item There exists a norm $\|\cdot\|$ such that $\|x\|^{2}=x \bar{x}=\bar{x} x$ with $\|x\|^{2}\|y\|^{2}=\|x y\|^{2}$.
\item $s \cdot x y=s x \cdot y=x s \cdot y=x \cdot s y=x \cdot y s$.
\item `Weak linearity' holds: $(x+s) y=x y+s y$ and $x(y+s)=x y+x s$.
\item Left distributivity holds: $x(y+z)=x y+x z$.
\item If $x \neq 0$, a unique $x^{-1}$ exists
satisfying $x^{-1} x=x x^{-1}=1$, equal to $x^{-1}=\bar{x}\|x\|^{-2}$. Inversion commutes with complex conjugation: $\overline{x^{-1}}=\bar{x}^{-1}$.
\item Left alternativity holds: $x \cdot x y=x^{2} y$.
\item Left cancellation holds: $x \cdot x^{-1} y=y$.
\end{enumerate}
\end{proposition}
Note the one-sidedness of many of the above properties. In particular, the 16-ons form a division left semialgebra, since all non-zero elements have an inverse but they lack the full distributivity properties required to be an algebra.\\
\indent We define $\mathcal{A}+\mathcal{A}$, $\mathcal{A}\mathcal{A}$, $\mathcal{A}\mathcal{A}^{-1}$, $E(\mathcal{A})$, and $E'(\mathcal{A})$ for finite subsets of the 16-ons identically as in the octonionic case. We denote the 16-ons by $\mathbb{S}$. The lemma bounding $E'(\mathcal{A})$ when $\mathcal{A}\subset \mathbb{O}$ adapts to the case of $\mathcal{A}\subset\mathbb{S}$ immediately, and we state it below.
\begin{lemma}
Let $\mathcal{A}\subset\mathbb{S}$ and $0\not\in\mathcal{A}$. Let $e_1,...,e_{|\mathcal{A}/\mathcal{A}|}$ be the elements of $\mathcal{A}/\mathcal{A}$, $a_i$ be the number of ways of multiplying an element of $\mathcal{A}$ with an element of $\mathcal{A}^{-1}$ (in that order) to obtain $e_i$, and $b_i$ be the number of ways of multiplying an element of $\mathcal{A}^{-1}$ with an element of $\mathcal{A}$ (in that order) to obtain $e_i$, for $i=1,...,|\mathcal{A}/\mathcal{A}|$. Suppose the $a_i$ and $b_i$ are such that $\sum_i a_ib_i\geq \operatorname{min}(\sum_i a_i^2,\sum_i b_i^2)$. Then $$E^\prime(\mathcal{A})\geq \frac{|\mathcal{A}|^4|\mathcal{A}/\mathcal{A}|}{\operatorname{min}\left(|\mathcal{A}\mathcal{A}^{-1}|,|\mathcal{A}^{-1}\mathcal{A}|\right)^2}.$$ When $\mathcal{A}$ is closed under taking inverses, $E^\prime(\mathcal{A})\geq \frac{|\mathcal{A}|^4}{|\mathcal{A}\mathcal{A}|}$.
\end{lemma}
\subsubsection{Niners}
The subset of the 16-ons comprised of elements with the last seven entries equal to zero form a useful subset of the 16-ons. This is so because they are the largest additively closed subset to satisfy both left and right distributivity, and they are also closed under inversion (not, however, multiplication). Both left and right distributivity appear crucial to the following argument. We prove this basic property below.
\begin{proposition}\label{ninerdistr}
   Let $x,y,z\in\mathbb{S}\setminus0$ and $x$ be a niner. Then $(y+z)x = yx+zx$.
\end{proposition}
\begin{proof}
    Write $x = (x_1,x_2)$ with $x_i\in\mathbb{O}$, $i=1,2$, and similary for $y$ and $z$. We have $x_2\in\mathbb{R}$. Then 
    \begin{align*} (y+z)x &= (y_1+z_1,y_2+z_2)(x_1,x_2)\\ &= \left((y_1+z_1) x_1-x_2 (\overline{y_2+z_2}), x_1 (y_2+z_2)+(\overline{y_1+z_1}) (y_2+z_2)^{-1} \cdot (y_2+z_2) x_2\right)\\
    &= \left(y_1x_1+z_1x_1-x_2\overline{y_2}-x_2\overline{z_2}, x_1y_2+x_1z_2+(\overline{y_1}+\overline{z_1}) x_2\right).
    \end{align*}
    Moreover, we have 
    \begin{align*}
        yx+zx &= (y_1,y_2)(x_1,x_2)+(z_1,z_2)(x_1,x_2)\\ &= \left(y_1 x_1-x_2 \overline{y_2}, x_1 y_2+\overline{y_1} y_2^{-1} \cdot y_2 x_2\right)+\left(z_1 x_1-x_2 \overline{z_2}, x_1 z_2+\overline{z_1} z_2^{-1} \cdot z_2 x_2\right)\\
        &= \left(y_1 x_1-x_2 \overline{y_2}+z_1 x_1-x_2 \overline{z_2}, x_1 y_2+\overline{y_1} y_2^{-1} \cdot y_2 x_2+x_1 z_2+\overline{z_1} z_2^{-1} \cdot z_2 x_2\right)\\
        &= \left(y_1 x_1-x_2 \overline{y_2}+z_1 x_1-x_2 \overline{z_2}, x_1 y_2+\overline{y_1}x_2+x_1 z_2+\overline{z_1}x_2\right).
    \end{align*}
\end{proof}
\section{Obtaining the Bound for the Octonions}
We follow the method of \cite{kissingsumprods}, in order to derive a sum-product bound for the octonions; the main technical difficulty is re-working the proofs therein to get around the restrictions imposed by nonassociativity.
To obtain the desired bound, it suffices to prove $$\frac{E^\prime(\mathcal{A})}{\log \lceil|\mathcal{A}|\rceil} \ll|\mathcal{A}+\mathcal{A}|^{2}.$$
For $x\in \mathcal{A}/\mathcal{A}$, let $\ell(x)$ ($r(x)$ respectively) denote the number of representatives of $x$ in $\mathcal{A}\mathcal{A}^{-1}$ ($\mathcal{A}^{-1}\mathcal{A}$ respectively).
We can then express ${E}^\prime(\mathcal{A})$ in terms of $\ell(x)$ and $r(x)$:
$$
E^\prime(\mathcal{A})=\sum_{x \in \mathcal{A} / \mathcal{A}} \ell(x) r(x)
$$
Without loss of generality, assume that
$$
\sum_{\substack{x \in \mathcal{A} / \mathcal{A} \\ \ell(x) \geq r(x)}} \ell(x) r(x) \geq \sum_{\substack{x \in \mathcal{A} / \mathcal{A} \\ \ell(x) \leq r(x)}} \ell(x) r(x).
$$
Then we can bound $E^\prime(\mathcal{A})$ as follows:
$$
E^\prime(\mathcal{A})=\sum_{x \in \mathcal{A} / \mathcal{A}} \ell(x) r(x) \leq 2 \sum_{\substack{x \in \mathcal{A} / \mathcal{A} \\ \ell(x) \geq r(x)}} \ell(x) r(x) \leq 2 \sum_{\substack{x \in \mathcal{A} / \mathcal{A} \\ \ell(x) \geq r(x)}} \ell(x)^{2} .
$$
Since $1 \leq \ell(x) \leq|\mathcal{A}|$, there exists an index $I$ such that elements of the set $\mathcal{R}:=\left\{x \in \mathcal{A} / \mathcal{A}: \ell(x) \geq r(x)\right.$ and $\left.2^{I} \leq \ell(x)<2^{I+1}\right\}$ contribute at least $1 /\left\lceil\log _{2}|\mathcal{A}|\right\rceil$ of the above sum. That is,
$$
\frac{E^\prime(\mathcal{A})}{2\lceil\log |\mathcal{A}|\rceil} \leq \sum_{x \in \mathcal{R}} \ell(x)^{2}<|\mathcal{R}| 2^{2 I+2}.
$$
For each $x \in \mathcal{R}$, we define a set $\mathcal{S}_{x}$ by
$$
\mathcal{S}_{x}=\left\{(a+c, b+d) \in(\mathcal{A}+\mathcal{A}) \times(\mathcal{A}+\mathcal{A}): a b^{-1}=x \text { and } c d^{-1}=\phi(x)\right\},
$$
where $\phi: \mathcal{R} \rightarrow \mathcal{R}$ maps an element of $\mathcal{R}$ to the closest (distinct) element in $\mathcal{R}$, that is
$$
\|x-\phi(x)\| \leq\|x-y\|
$$
for all $x, y \in \mathcal{R}$. We show that $\bigcup_{s \in \mathcal{R}} \mathcal{S}_{x}$ has sufficiently large size in $(\mathcal{A}+\mathcal{A}) \times(\mathcal{A}+\mathcal{A})$ (i.e. comprises a proportion larger than $E(\mathcal{A}) /\lceil\log |\mathcal{A}|\rceil$, multiplied by some constant factor).\\
\indent In the proofs of the next two lemmas, we use the fact that while the octonions are not associative, they satisfy weaker identities than associativity. In particular, if $x,y\in\mathbb{O}$, then $y = (yx)x^{-1} = x^{-1}(xy)$ \cite{conwayoctsquats}. The following lemma guarantees that
$$
\left|\mathcal{S}_{x}\right| \geq \ell(x) \ell(\phi(x)) \geq 2^{2 I}.
$$
\begin{lemma}
Given $p,q\in\mathcal{A}+\mathcal{A}$ and distinct $x,y\in\mathcal{A}\mathcal{A}^{-1}$, there is at most one quadruple $(a,b,c,d)\in\mathcal{A}^4$ such that $$a+c = p,\quad b+d = q,\quad ab^{-1} = x,\quad cd^{-1}=y.$$
\end{lemma}
\begin{proof}
Since $x$ and $y$ are distinct, their difference is invertible. Then solving the equations $$d = (y-x)^{-1}(p-xq),\quad b = q-d,\quad c = yd,\quad a = xb$$ yields a unique tuple $(a,b,c,d)$, if such a solution exists.
\end{proof}

\begin{lemma}
Let $a,b,c$ and $d$ be octonions, and $b$ and $d$ be non-zero and lie in the same hexapentacontadictant. Then $$\|(a+c)(b+d)^{-1}-ab^{-1}\|\leq\|cd^{-1}-ab^{-1}\|.$$
\end{lemma}
\begin{proof}
We have \begin{align*}
    \|(a+c)(b+d)^{-1}-ab^{-1}\| &= \|a+c-(ab^{-1})(b+d)\|\frac{1}{\|b+d\|}\\
    &= \|a+c - (ab^{-1})b-(ab^{-1})d\|\frac{1}{\|b+d\|}\\ &= \|a+c-a-(ab^{-1})d\|\frac{1}{\|b+d\|}\\ &= \|cd^{-1}-ab^{-1}\|\frac{\|d\|}{\|b+d\|}\\ &\leq \|cd^{-1}-ab^{-1}\|.
\end{align*}
\end{proof}
\begin{lemma}
Let $P\in\mathbb{R}^8$ be contained in $m$ closed balls, denoted $\mathcal{B}_i$, each with center $Q_i\in\mathbb{R}^8$, for $i=1,...,m$. Suppose $Q_i$ is not contained in the interior of any $\mathcal{B}_j$, for $j\ne i$. Then $m\leq 241$.
\end{lemma}
\begin{proof}
In the same manner as the proof of \cite[Lemma 8]{kissingsumprods}, one can construct $m-1$ non-overlapping spheres of equal radius, all touching some central sphere. Since the maximum number of spheres that can be arranged in this configuration in dimension 8 is 240 (the so-called kissing number), we arrive at the result.
\end{proof}
Combining the above results, we have 
$$|\mathcal{A}+\mathcal{A}|^{2}\geq\left|\bigcup_{x \in \mathcal{R}} \mathcal{S}_{x}\right| \geq \frac{1}{241} \sum_{x \in \mathcal{R}}\left|\mathcal{S}_{x}\right| \geq \frac{1}{241}|\mathcal{R}| 2^{2 I} \geq \frac{E^{\prime}(\mathcal{A})}{1928\lceil\log |\mathcal{A}|\rceil},$$
which implies the required bound.
\section{Obtaining a Bound for the 16-ons}
In this section we adapt the analysis carried out for the octonions for niner subsets of the 16-ons. The weaker properties satisfied by the elements of $\mathbb{S}$ introduce some subtleties into the proofs and force the restriction to sets of niners (which satisfy good distributive identities), but the argument proceeds largely as before. Having completed this, we rearrange the final inequalities to get a result lower bounding $k_{16}$.
\subsubsection{Hexatriacontapentactapentaliahexacismyriants} We will need to divide the 16-dimensional plane into parts depending on the coefficients of a given 16-on. Instead of using increasingly lengthy numerical prefixes, we will simply refer to the portions of the plane containing 16-ons all of whose coordinates in the vector representation in $\mathbb{R}^{16}$ are positive, or all negative, as the \textit{privileged orthants}. That is, for a 16-on $x=\alpha + \beta i + \gamma j +...\in \mathbb{S}$, we can write $x=(\alpha,\beta,\gamma,...)\in\mathbb{R}^{16}$, and the privileged orthants contain the 16-ons with vectors whose entries are all positive, or all negative.\\

\noindent To obtain the bound for the 16-ons, we restrict to subsets $\mathcal{A}$ of niners, and it again suffices to prove $$\frac{E(\mathcal{A})}{\log \lceil|\mathcal{A}|\rceil} \ll|\mathcal{A}+\mathcal{A}|^{2}.$$
We define $\ell(x)$ and $r(x)$ as before and express ${E}^\prime(\mathcal{A})$ as
$$
E^\prime(\mathcal{A})=\sum_{x \in \mathcal{A} / \mathcal{A}} \ell(x) r(x),
$$
and assume that
$$
\sum_{\substack{x \in \mathcal{A} / \mathcal{A} \\ \ell(x) \geq r(x)}} \ell(x) r(x) \geq \sum_{\substack{x \in \mathcal{A} / \mathcal{A} \\ \ell(x) \leq r(x)}} \ell(x) r(x).
$$
The bound on $E^\prime(\mathcal{A})$
$$
E^\prime(\mathcal{A})=\sum_{x \in \mathcal{A} / \mathcal{A}} \ell(x) r(x) \leq 2 \sum_{\substack{x \in \mathcal{A} / \mathcal{A} \\ \ell(x) \geq r(x)}} \ell(x) r(x) \leq 2 \sum_{\substack{x \in \mathcal{A} / \mathcal{A} \\ \ell(x) \geq r(x)}} \ell(x)^{2}
$$
follows. Since $1 \leq \ell(x) \leq|\mathcal{A}|$, there is an index $I$ such that elements of $\mathcal{R}:=\left\{x \in \mathcal{A} / \mathcal{A}: \ell(x) \geq r(x)\right.$ and $\left.2^{I} \leq \ell(x)<2^{I+1}\right\}$ contribute at least $1 /\left\lceil\log _{2}|\mathcal{A}|\right\rceil$ of the above sum. Thus
$$
\frac{E^\prime(\mathcal{A})}{2\lceil\log |\mathcal{A}|\rceil} \leq \sum_{x \in \mathcal{R}} \ell(x)^{2}<|\mathcal{R}| 2^{2 I+2}.
$$
For $x \in \mathcal{R}$, we define $\mathcal{S}_{x}$ by
$$
\mathcal{S}_{x}=\left\{(a+c, b+d) \in(\mathcal{A}+\mathcal{A}) \times(\mathcal{A}+\mathcal{A}): b^{-1}a=x \text { and } d^{-1}c=\phi(x)\right\},
$$
where $\phi: \mathcal{R} \rightarrow \mathcal{R}$ is a map sending an element to a closest element in $\mathcal{R}$, i.e.
$$
\|x-\phi(x)\| \leq\|x-y\|
$$
for all $x, y \in \mathcal{R}$. We show that $\bigcup_{s \in \mathcal{R}} \mathcal{S}_{x}\subset (\mathcal{A}+\mathcal{A}) \times(\mathcal{A}+\mathcal{A})$ has size $\gg E(\mathcal{A}) /\lceil\log |\mathcal{A}|\rceil$.\\
\indent In the proofs of the next two lemmas, we use the fact that while the 16-ons are not associative, they satisfy weaker identities than associativity. In particular, we have $x^{-1}(xy)=y$. Moreover, left multiplication of sedenions is a linear operation, i.e. $a(b+c) = ab+ac$ for $a,b,c\in\mathbb{S}$. However, we note that for arbitrary 16-ons, right multiplication is not linear. This is the reason we restrict to subsets of niners; by Proposition \ref{ninerdistr}, the equation $(u+v)w = uw+vw$ holds when $u,v\in\mathbb{S}$ and $w$ is a niner.  The following lemma implies
$$
\left|\mathcal{S}_{x}\right| \geq \ell(x) \ell(\phi(x)) \geq 2^{2 I}.
$$
\begin{lemma}
Given a finite set $\mathcal{A}$ of niners, $p,q\in\mathcal{A}+\mathcal{A}$ and distinct $x,y\in\mathcal{A}^{-1}\mathcal{A}$, there is at most one quadruple $(a,b,c,d)\in\mathcal{A}^4$ such that $$a+c = p,\quad b+d = q,\quad b^{-1}a = x,\quad d^{-1}c=y.$$
\end{lemma}
\begin{proof}
Since $x$ and $y$ are distinct, their difference is invertible. Then solving the equations $$d = (y-x)^{-1}(p-xq),\quad b = q-d,\quad c = dy,\quad a = bx$$ yields a unique tuple $(a,b,c,d)$, if such a solution exists.
\end{proof}

\begin{lemma}\label{ninerballs}
Let $a,b,c$ and $d$ be niners such that $(a+c,b+d)\in\mathcal{S}_x$ for some $x$, and $b$ and $d$ be non-zero and lie in the same privileged orthant. Then $$\|(b+d)^{-1}(a+c)-b^{-1}a\|\leq\|d^{-1}c-b^{-1}a\|.$$
\end{lemma}
\begin{proof}

We have 
\begin{align*}
    \|(b+d)^{-1}(a+c)-b^{-1}a\| &= \|a+c-(b+d)(b^{-1}a)\|\frac{1}{\|b+d\|}\\
    &= \|a+c-b(b^{-1}a)-d(b^{-1}a)\|\frac{1}{\|b+d\|}\\
    &= \|d^{-1}c-b^{-1}a\|\frac{\|d\|}{\|b+d\|}.
\end{align*}
\end{proof}

The above lemma says that if $(a+c,b+d)\in S_x$, there is a 16-on $(b+d)^{-1}(a+c)$ lying in the closed ball centered at $x$ of radius $\phi(x)$. The following lemma shows that there cannot be too many 16-ons lying in any given such closed ball.

\begin{lemma}
Let $P\in\mathbb{R}^{16}$ be contained in $m$ closed balls, denoted $\mathcal{B}_i$, each with center $Q_i\in\mathbb{R}^{16}$, for $i=1,...,m$. Suppose $Q_i$ is not contained in the interior of any $\mathcal{B}_j$, for $j\ne i$. Then $m\leq 7333$.
\end{lemma}
\begin{proof}
As in the proof of \cite[Lemma 8]{kissingsumprods}, one can construct $m-1$ non-overlapping spheres of equal radius, all touching some central sphere. Since the maximum number of spheres that can be arranged in this configuration in dimension 16 is (currently) bounded between 4320 and 7332 \cite{kissingupperbd}, we arrive at the result.
\end{proof}

\noindent As in the octonionic case, we then obtain
\begin{equation}\label{equation2}
|\mathcal{A}+\mathcal{A}|^{2}\geq\left|\bigcup_{x \in \mathcal{R}} \mathcal{S}_{x}\right| \geq \frac{1}{7333} \sum_{x \in \mathcal{R}}\left|\mathcal{S}_{x}\right| \geq \frac{1}{7333}|\mathcal{R}| 2^{2 I} \gg \frac{E(\mathcal{A})}{\lceil\log |\mathcal{A}|\rceil}.
\end{equation}
\subsubsection{Bounding $k_{16}$}
Note that the inequalities $(\ref{equation2})$ can be rewritten $$|\mathcal{A}+\mathcal{A}|^{2}\geq\left|\bigcup_{x \in \mathcal{R}} \mathcal{S}_{x}\right| \geq \frac{1}{k_{16}+1} \sum_{x \in \mathcal{R}}\left|\mathcal{S}_{x}\right| \geq \frac{1}{k_{16}+1}|\mathcal{R}| 2^{2 I} \gg \frac{E(\mathcal{A})}{(8k_{16}+8)\lceil\log |\mathcal{A}|\rceil}.$$

\noindent Rearranging, we obtain $k_{16}\geq \frac{E(\mathcal{A})}{8\lceil\log |\mathcal{A}|\rceil|\mathcal{A}+\mathcal{A}|^2}-1$, and more tightly, $$k_{16}\geq \frac{\sum_{x \in \mathcal{R}}\left|\mathcal{S}_{x}\right|}{\left|\bigcup_{x \in \mathcal{R}} \mathcal{S}_{x}\right|}-1.$$ 
Therefore by choosing sets $\mathcal{A}\subset \mathbb{S}$ carefully, one may be able to establish concrete lower bounds on $k_{16}$. Stated precisely, we have
\begin{theorem}
    Let $\mathcal{A}\subset\mathbb{S}$ be a finite set of niners such that $0\not\in\mathcal{A}$. Let $\mathcal{R}$ and $S_x$ be defined as above. Then $$k_{16}\geq \frac{\sum_{x \in \mathcal{R}}\left|\mathcal{S}_{x}\right|}{\left|\bigcup_{x \in \mathcal{R}} \mathcal{S}_{x}\right|}-1.$$    
\end{theorem}
\section{Conclusion}
We have obtained sum-product bounds, firstly for subsets of the octonions, and secondly for subsets of niners within the 16-ons, this latter result implying a bound on the kissing number $k_{16}$. The primary outstanding question from this work is whether a concrete lower bound on $k_{16}$ can be obtained by selecting particular sets of niners.
\section{Acknowledgements}
The author would like to thank Edmund Dable-Heath for useful conversations.

\end{document}